\documentclass[final,onefignum,onetabnum]{siamart251216}

\usepackage{amsmath,amsfonts,amsopn,amssymb}
\usepackage{graphicx}
\usepackage[caption=false]{subfig}
\usepackage{multirow,relsize,makecell}
\usepackage{url}
\ifpdf
\DeclareGraphicsExtensions{.eps,.pdf,.png,.jpg}
\else
\DeclareGraphicsExtensions{.eps}
\fi
\usepackage{algorithm, algorithmic}
\usepackage[shortlabels]{enumitem}
\usepackage{accents}

\numberwithin{theorem}{section}
\newsiamremark{assumption}{Assumption}
\newsiamremark{remark}{Remark}
\newsiamremark{question}{Question}
\newsiamremark{example}{Example}
\crefname{assumption}{Assumption}{Assumptions}
\crefname{remark}{Remark}{Remarks}
\crefname{example}{Example}{Examples}

\headers{Variational estimates for Schur complements}{Jongho Park}

\title{Variational estimates for Schur complements\\ of symmetric saddle point problems\thanks{Submitted to arXiv.
\funding{This work was supported by the University Development Fund from The Chinese University of Hong Kong,
Shenzhen (UDF01004630).}
}}

\author{
Jongho Park\thanks{School of Science and Engineering, The Chinese University of Hong Kong, Shenzhen, Guangdong 518172, China
 (\email{jonghopark@cuhk.edu.cn}).}
}

\ifpdf
\hypersetup{
  pdftitle={Variational estimates for Schur complements of symmetric saddle point problems},
  pdfauthor={Jongho Park}
}
\fi

\usepackage[normalem]{ulem}
\usepackage{color}

\begin{document}

\maketitle
\begin{abstract}
We present sharp variational estimates for the extremal eigenvalues of Schur complements arising in symmetric saddle point problems.
These estimates make explicit the role of right inverses with controlled energy in lower bounds and of null space corrections in upper bounds.
The same argument applies to projected Schur complements when the primal operator is semidefinite.
Applications to the augmented Lagrangian method, mixed finite element methods, and nonoverlapping domain decomposition methods yield concise proofs of established convergence estimates.
\end{abstract}



\section{Introduction}
\label{Sec:Introduction}
This paper is concerned with iterative methods for solving the symmetric linear saddle point problem defined on finite-dimensional vector spaces~$V$ and $W$:
\begin{equation}
\label{model_saddle}
\begin{bmatrix}
    A & B^t \\ B & 0 
\end{bmatrix}
\begin{bmatrix}
    u \\ p
\end{bmatrix}
=
\begin{bmatrix}
    f \\ g
\end{bmatrix},
\end{equation}
where $A \colon V \to V$ is a linear operator that is either symmetric positive definite (SPD) or symmetric positive semidefinite (semi-SPD), $B \colon V \to W$ is a surjective linear operator, and $f \in V$, $g \in W$ are given data. 
We denote by $(u,p)$ the solution of~\eqref{model_saddle}. 
It is well-known that the system~\eqref{model_saddle} is equivalent to the constrained quadratic optimization problem
\begin{equation}
\label{model_saddle_opt}
\min_{v \in V} \left\{ \frac{1}{2}(Av,v) - (f,v) \right\}
\quad \text{subject to} \quad
Bv = g,
\end{equation}
where $p$ serves as the Lagrange multiplier enforcing the constraint.
Such saddle point problems arise in many areas of scientific computing; see~\cite{BGL:2005} and the references therein.

There are three principal approaches for iterative methods for solving the saddle point problem~\eqref{model_saddle}~\cite{BGL:2005,Xu:2010}. 
The first is to apply iterative methods to the SPD dual problem of~\eqref{model_saddle} obtained in terms of the Schur complement. 
The second is to apply stationary iterations, such as Uzawa-type methods~\cite{BPV:1997} or augmented Lagrangian methods~\cite{FG:1983}. 
The third is to apply Krylov subspace methods, such as MINRES~\cite{PS:1975}, directly to~\eqref{model_saddle} with suitable preconditioners~\cite{MGW:2000}. 
In all three approaches, the properties of the Schur complement play a central role in both the design and the analysis of algorithms.

The purpose of this paper is to present sharp estimates for the extremal eigenvalues of the Schur complements arising in symmetric saddle point problems of the form~\eqref{model_saddle}.
Although these estimates can be derived by elementary arguments, they provide useful insight into the analysis of Schur complements arising in various applications.
For the minimum eigenvalue, the variational expression highlights the importance of choosing a right inverse of $B$ with controlled energy to obtain a sharper lower bound.
For the maximum eigenvalue, suitable corrections from $\mathcal{N}(B)$ leave $Bv$ unchanged and can yield sharper upper bounds.
These observations give an explicit variational interpretation of techniques used in existing analyses of Schur complements~\cite{BBF:2013,MS:2007,MT:2001}.

Estimates for the extremal eigenvalues of the Schur complement when $A$ is SPD are presented in \cref{Thm:Schur} and are sharp in the sense that they are expressed as identities. 
Sharp estimates for the Schur complement preconditioned by a general SPD preconditioner are given in \cref{Thm:Schur_preconditioned}.  
Finally, our results extend naturally to the case when $A$ is semi-SPD, as shown in \cref{Thm:Schur_singular,Thm:Schur_singular_preconditioned}.
The semi-SPD extension shows that, once the null space and the compatibility condition have been accounted for, the singularity of the primal operator introduces no additional difficulty in the spectral analysis.

To demonstrate the usefulness of the proposed estimates for the Schur complement, we present several applications from scientific computing. 
The first example is the augmented Lagrangian method~\cite{FG:1983}, a stationary iterative method for solving~\eqref{model_saddle}. 
We provide sharp estimates for the error propagation operator of the augmented Lagrangian method. 
In particular, we recover the results of~\cite{LP:2009,LP:2017,LWXZ:2007} showing that the convergence rate of the augmented Lagrangian method can be made arbitrarily fast by choosing the augmented Lagrangian parameter $\epsilon$ sufficiently small. 
The second example is mixed finite element methods for solving mixed variational problems~\cite{BBF:2013}. 
We give concise analyses of Schur complements arising in mixed finite element discretizations of Darcy flow~\cite{BOS:2018,PH:2025} and the Stokes equations~\cite{ABMXZ:2014,Verfurth:1984}, emphasizing the role of right inverses provided by the discrete inf--sup condition.
The final example concerns nonoverlapping domain decomposition methods~\cite{TW:2005,XZ:1998}, specifically FETI (finite element tearing and interconnecting)~\cite{FR:1991} and FETI-DP (dual--primal FETI)~\cite{FLLPR:2001,KWD:2002}. 
For both unpreconditioned and preconditioned versions of FETI and FETI-DP, we illustrate how right inverses and null space corrections enter the convergence estimates.

The remainder of this paper is organized as follows. 
In \cref{Sec:Auxiliary}, we present variational estimates used in the subsequent analysis.
In \cref{Sec:Saddle}, we review well-posedness, Schur complement systems, and strategies for designing iterative methods for symmetric saddle point problems. 
In \cref{Sec:Schur}, we present sharp estimates for the extremal eigenvalues of the Schur complements. 
In \cref{Sec:ALM,Sec:Mixed,Sec:FETI}, we demonstrate applications of our results to problems arising in scientific computing, including the augmented Lagrangian method, mixed finite element methods, and nonoverlapping domain decomposition methods. 
Finally, in \cref{Sec:Conclusion}, we conclude with some remarks.

\subsection{Notation}
We summarize here the notation used throughout the paper. 
Let $V$ and $W$ be finite-dimensional vector spaces equipped with inner products $(\cdot,\cdot)$ and the associated norms $\|\cdot\|$. 
For any subspace $X \subset V$, we denote by $X^\perp$ its orthogonal complement. 
For a linear operator $L \colon V \to W$, the null space and the range are denoted by $\mathcal{N}(L)$ and $\mathcal{R}(L)$, respectively, and the adjoint of $L$ is denoted by $L^t \colon W \to V$. 

Following standard convention in the literature~(see, e.g.,~\cite{Xu:1992}), we write $x \lesssim y$, or equivalently $y \gtrsim x$, to mean that there exists a constant $C>0$, independent of the relevant parameters, such that $x \leq C y$. 
Furthermore, we write $x \eqsim y$ when both $x \lesssim y$ and $x \gtrsim y$ hold.

\section{Variational estimates}
\label{Sec:Auxiliary}

In this section, we present a variational identity and the resulting estimates for the extremal eigenvalues of preconditioned operators.

We begin with an elementary identity that has proved useful in the design and analysis of various advanced numerical methods.
This identity is closely related to the fictitious space method of Nepomnyaschikh~\cite[Lemma~2.3]{Nepomnyaschikh:1992} and the auxiliary space method of Xu~\cite{Xu:1996}.
An additive version of the identity is used in~\cite{GO:1995,XZ:2002}, and related formulations can be found in~\cite{ABMXZ:2014,BS:2007,Chen:2011,PX:2025}.
For completeness, we provide an elementary proof using a saddle point problem.

\begin{lemma}
\label{Lem:aux_space_lemma}
Let $V$ and $\undertilde{V}$ be finite-dimensional vector spaces and let $\Pi \colon \undertilde{V} \to V$ be a surjective linear operator.
Let $\undertilde{M} \colon \undertilde{V} \to \undertilde{V}$ be an SPD linear operator, and define
\begin{equation}
\label{B_aux}
M = \Pi \undertilde{M} \Pi^t \colon V \to V.
\end{equation}
Then $M$ is SPD.
Moreover, it satisfies
\begin{equation*}
( M^{-1} v, v ) = \inf_{\undertilde{v} \in \undertilde{V},\ \Pi \undertilde{v} = v } ( \undertilde{M}^{-1} \undertilde{v}, \undertilde{v} ),
\quad v \in V.
\end{equation*}
\end{lemma}

\begin{proof}
Since $\Pi$ is surjective, $\Pi^t$ is injective, and hence $M$ is SPD.
For any $v \in V$, it follows that
\begin{align*}
\inf_{\undertilde{v} \in \undertilde{V},\ \Pi \undertilde{v} = v}
    (\undertilde{M}^{-1}\undertilde{v},\undertilde{v})
& = \inf_{\undertilde{v} \in \undertilde{V}} \sup_{w \in V}
    \left\{ (\undertilde{M}^{-1}\undertilde{v},\undertilde{v}) + 2(w,v-\Pi\undertilde{v}) \right\} \\
&= \sup_{w \in V} \inf_{\undertilde{v} \in \undertilde{V}}
    \left\{ (\undertilde{M}^{-1}\undertilde{v},\undertilde{v}) + 2(w,v-\Pi\undertilde{v}) \right\} \\
&= \sup_{w \in V} \left\{ 2(w,v) - (Mw,w) \right\} \\
    &= (M^{-1}v,v).
\end{align*}
The interchange of the infimum and supremum is justified by the saddle point
$(\undertilde{v},w) = (\undertilde{M}\Pi^t M^{-1}v,M^{-1}v)$ of the quadratic expression.
\end{proof}

Now, let $L \colon V \to V$ and $M \colon V \to V$ be SPD linear operators, and consider the situation where $M$ is used as a preconditioner for the linear system associated with $L$. 
In this case, the convergence rate of iterative methods can be estimated in terms of the condition number of the preconditioned operator $M L$, given by
\begin{equation*}
    \kappa(M L) = \frac{\lambda_{\max}(M L)}{\lambda_{\min}(M L)},
\end{equation*}
where $\lambda_{\min}$ and $\lambda_{\max}$ denote the minimum and maximum eigenvalues, respectively.

The following theorem is a straightforward consequence of \cref{Lem:aux_space_lemma} and the Courant--Fischer min--max theorem~\cite{GV:2013}.

\begin{theorem}
\label{Thm:condition_number_aux}
Let $L \colon V \to V$ be an SPD linear operator, and let $M \colon V \to V$ be given in~\eqref{B_aux}.
Then we have
\begin{align*}
\lambda_{\min} (M L) &= \left( \sup_{ v \in V,\ \| v \|_{L} = 1} \inf_{\undertilde{v} \in \undertilde{V},\ \Pi \undertilde{v} = v } ( \undertilde{M}^{-1} \undertilde{v}, \undertilde{v} ) \right)^{-1}, \\
\lambda_{\max} (M L) &= \left( \inf_{ v \in V,\ \| v \|_{L} = 1} \inf_{ \undertilde{v} \in \undertilde{V},\ \Pi \undertilde{v} = v } ( \undertilde{M}^{-1} \undertilde{v}, \undertilde{v} ) \right)^{-1}.
\end{align*}
\end{theorem}

\begin{proof}
Since $(ML)^{-1}$ is SPD with respect to the inner product $(L\cdot,\cdot)$, the Courant--Fischer min--max theorem gives
\begin{equation*}
\lambda_{\min}(ML)^{-1}
= \sup_{v \in V,\ \|v\|_L=1} (M^{-1}v,v), \quad
\lambda_{\max}(ML)^{-1}
= \inf_{v \in V,\ \|v\|_L=1} (M^{-1}v,v).
\end{equation*}
Applying \cref{Lem:aux_space_lemma} yields the result.
\end{proof}

The result in Theorem~\ref{Thm:condition_number_aux} can be extended to more general situations, for instance when $L$ is only semi-SPD, $M$ is a general linear operator, and $\undertilde{M}$ is semi-SPD. 
Further details on these extensions can be found in~\cite{PX:2025}.

\section{Symmetric saddle point problems}
\label{Sec:Saddle}
For completeness, we review the essentials of symmetric saddle point problems: the problem statement, well-posedness, the Schur complement system, and strategies for designing iterative methods.
For comprehensive literature on saddle point problems, we refer the reader to, e.g.,~\cite{BGL:2005,BBF:2013}.

\subsection{Model problem}
We consider the model symmetric saddle point problem~\eqref{model_saddle}. 
Since $B$ is surjective, we have (see~\cite{XZ:2003})
\begin{equation}
\label{Brezzi}
\inf_{q \in W,\ \| q \| = 1} \sup_{v \in V,\ \| v \| = 1}
(Bv, q) = \| B^{-1} \|^{-1} > 0,
\end{equation}
where $B$ is regarded as an isomorphism from $\mathcal{N}(B)^{\perp}$ onto $W$. 
In other words, the Babu\v{s}ka--Brezzi (inf--sup) condition~\cite{BBF:2013} holds.

Moreover, under the surjectivity assumption on $B$, the system operator
\begin{equation*}
\mathcal{A} = 
\begin{bmatrix}
    A & B^t \\ B & 0 
\end{bmatrix}
\colon V \times W \to V \times W
\end{equation*}
associated with the problem~\eqref{model_saddle} is nonsingular if and only if
\begin{equation}
\label{well-posed}
\mathcal{N}(A) \cap \mathcal{N}(B) = \{0\},
\end{equation}
as established in~\cite[Theorem~3.4]{BGL:2005}.
In what follows, we always assume that the condition~\eqref{well-posed} holds to ensure the well-posedness of~\eqref{model_saddle}.

\subsection{Elimination of the primal variable}
From the saddle point system~\eqref{model_saddle}, one can eliminate the primal variable $u$ and obtain a reduced system formulated solely in terms of the dual variable $p$. 
Assume first that $A$ is SPD. 
From the first equation in~\eqref{model_saddle}, we have
\begin{equation}
\label{u_elimination}
u = A^{-1}(f - B^t p),
\end{equation}
and substituting~\eqref{u_elimination} into the second equation of~\eqref{model_saddle} yields the dual problem
\begin{equation}
\label{Schur}
S p = d,
\quad \text{where} \quad
S = B A^{-1} B^t \colon W \to W, \quad
d = B A^{-1} f - g \in W.
\end{equation}
In~\eqref{Schur}, $S$ is called the Schur complement. 
Because $A$ is SPD and $B$ is surjective, $S$ is of the form~\eqref{B_aux}. 
In particular, by \cref{Lem:aux_space_lemma}, we deduce that $S$ is SPD. 
Hence, the saddle point system~\eqref{model_saddle} is equivalent to the SPD linear system~\eqref{Schur} for $p$, followed by a back substitution to recover $u$.

The case when $A$ is only semi-SPD is more delicate, since~\eqref{u_elimination} is not valid. 
Nevertheless, elimination of the primal variable $u$ is still possible when full information about the null space $\mathcal{N}(A)$ is available, that is, when an explicit basis of $\mathcal{N}(A)$ is known. 
Such a situation arises, for example, in Neumann boundary value problems for the Poisson equation and linear elasticity~\cite{FR:1991}.

In the following, we summarize the arguments from~\cite[Section~2.2]{Pechstein:2013} and~\cite[Section~6.3]{TW:2005} concerning the elimination of $u$. 
Let
\[
N \colon \mathbb{R}^{\dim \mathcal{N}(A)} \to \mathcal{N}(A)
\]
be an injective linear operator with range $\mathcal{N}(A)$. 
Such a map $N$ can be constructed as the matrix whose columns form a basis of $\mathcal{N}(A)$. 
From the first equation of~\eqref{model_saddle}, compatibility requires $f - B^t p \in \mathcal{R}(A)$, which is equivalent to
\begin{equation}
\label{compatibility}
N^t(f - B^t p) = 0.
\end{equation}
In this case, $u$ satisfies
\begin{equation}
\label{u_elimination_singular}
u = A^{+}(f - B^t p) + N\xi,
\end{equation}
for some $\xi \in \mathbb{R}^{\dim \mathcal{N}(A)}$, where $A^{+} \colon V \to V$ denotes a pseudoinverse of $A$, i.e., $AA^+$ acts as the identity on $\mathcal{R} (A)$. 
Substituting~\eqref{u_elimination_singular} into the second equation of~\eqref{model_saddle} yields
\begin{equation}
\label{Schur_singular_intermediate}
S p - B N \xi = d, \quad \text{where} \quad
S = B A^+ B^t \colon W \to W,\
d = B A^+f - g \in W.
\end{equation}
Note that~\eqref{Schur_singular_intermediate} reduces to~\eqref{Schur} when $A$ is SPD.

Now,~\eqref{compatibility} and~\eqref{Schur_singular_intermediate} form a coupled system in the variables $p$ and $\xi$. 
We next eliminate $\xi$. 
Since $\mathcal{R} (N) =\mathcal{N}(A)$, the condition~\eqref{well-posed} implies that $BN$ is injective. 
Hence the orthogonal projection onto
\begin{equation}
\label{W0}
W_0 := \mathcal{R} (B|_{\mathcal{N}(A)})^{\perp} = \mathcal{R}(BN)^{\perp} \subset W
\end{equation}
is well defined and given by
\begin{equation}
\label{P}
P = I - BN ((BN)^t BN)^{-1} (BN)^t \colon W \to W_0.
\end{equation}
The computation of $P$ requires solving a linear system involving $(BN)^t BN$, which can be carried out at negligible cost when $\dim \mathcal{N}(A)$ is small.

We decompose $p$ orthogonally into components in $W_0$ and $W_0^{\perp} = \mathcal{R} (B|_{\mathcal{N}(A)})$:
\begin{equation}
\label{p_decomposition}
    p = P^t p_0 + BN \eta, 
    \quad p_0 \in W_0,\ 
    \eta \in \mathbb{R}^{\dim \mathcal{N}(A)}.
\end{equation}
Substituting~\eqref{p_decomposition} into~\eqref{compatibility} gives
\begin{equation*}
    \eta = ((BN)^t BN )^{-1} N^t f.
\end{equation*}
Note that $\eta$ can be computed prior to the main iterations of an iterative method. 
Finally, substituting~\eqref{p_decomposition} into~\eqref{Schur_singular_intermediate} and applying $P$ to both sides yields
\begin{equation}
\label{Schur_singular}
    S_0 p_0 = d_0,
    \quad \text{where} \quad
    S_0 = P S P^t \colon W_0 \to W_0,\ d_0 = P(d - SBN \eta ) \in W_0.
\end{equation}
Thus, we arrive at a linear system in the single unknown $p_0$, which represents the dual problem in the case when $A$ is semi-SPD. 
\cref{Prop:Schur_singular} shows that $S_0$ is SPD and independent of the particular choice of the pseudoinverse $A^+$~(cf.~\cite[Lemma~2.19]{Pechstein:2013}).

\begin{proposition}
\label{Prop:Schur_singular}
The projected Schur complement $S_0 \colon W_0 \to W_0$ given in~\eqref{Schur_singular} is independent of the particular choice of the pseudoinverse $A^+$.
Moreover, $S_0$ is SPD.
\end{proposition}
\begin{proof}
Since $S_0 = PB A^+ (PB)^t$, to prove independence of the choice of $A^+$, it suffices to show that $PB|_{\mathcal{N}(A)} = 0$, which follows directly from the definition~\eqref{P} of $P$. 
Meanwhile, to prove that $S_0$ is SPD, it is enough, by \cref{Lem:aux_space_lemma}, to verify that $PB|_{\mathcal{R}(A)}$ is surjective. 
This holds because $PB$ is surjective and $PB|_{\mathcal{N}(A)} = 0$.
\end{proof}

\begin{remark}
\label{Rem:P}
While we define $P$ as the orthogonal projection onto $W_0$ in~\eqref{P}, it is also possible to employ oblique projections, which are often more effective for heterogeneous problems; see~\cite{Pechstein:2013,TW:2005}. 
However, we omit the details here.
\end{remark}

\begin{remark}
\label{Rem:FR}
The dual problems~\eqref{Schur} and~\eqref{Schur_singular} are special cases of Fenchel--Rockafellar duality in convex optimization, since~\eqref{model_saddle} is equivalent to the quadratic optimization problem~\eqref{model_saddle_opt}. 
For further details on Fenchel--Rockafellar duality, we refer the reader to~\cite{JPX:2025}.
\end{remark}

\subsection{Iterative methods}
Here, we discuss iterative methods for solving the saddle point problem~\eqref{model_saddle}.
For simplicity, we assume that $A$ is SPD.
Based on the surveys~\cite{BGL:2005,Xu:2010}, there are three major approaches to iterative methods.

The first approach is to apply iterative methods to the SPD dual problem~\eqref{Schur} (\eqref{Schur_singular} in the semi-SPD case).
In this case, we can utilize a large body of existing results on iterative methods for SPD linear systems.
See~\cite{PX:2025} for a systematic approach to the design and analysis of iterative methods for SPD linear systems.

The second approach is to apply stationary iterations, such as Uzawa-type methods~\cite{BPV:1997} or the augmented Lagrangian method~\cite{FG:1983}, to~\eqref{model_saddle}. 
It is well-known that the classical Uzawa method and the augmented Lagrangian method are equivalent to certain iterative schemes for solving the dual problem~\eqref{Schur}~\cite{BGL:2005}. 
Moreover, for general inexact Uzawa methods, their convergence behavior is determined by properties of the Schur complement~\cite{BPV:1997}.

The third approach is to apply preconditioned Krylov subspace methods, such as MINRES~\cite{PS:1975}, to~\eqref{model_saddle}. 
It was shown in~\cite{MGW:2000} that the block-diagonal preconditioner
\begin{equation*}
\mathcal{L} = 
\begin{bmatrix} 
A^{-1} & 0 \\ 
0 & S^{-1} 
\end{bmatrix}
\end{equation*}
is optimal in the sense that the preconditioned operator $\mathcal{L}\mathcal{A}$ has at most three distinct eigenvalues, namely $1$ and $\tfrac{1}{2}(1 \pm \sqrt{5})$. 
Thus, in this case, the central task is the construction of efficient preconditioners for $A$ and $S$~\cite{BGL:2005}.

In all three approaches, the properties of the Schur complement $S$ are central to both design and analysis of algorithms. 
Hence, in the remainder of this paper, we focus on the analysis of $S$. 
In particular, we use the variational estimates in \cref{Sec:Auxiliary} to analyze its extremal eigenvalues.

\section{Estimates for the Schur complements}
\label{Sec:Schur}
In this section, we derive variational identities for the extremal eigenvalues of the Schur complements given in~\eqref{Schur} and~\eqref{Schur_singular}, and use them to recover several classical bounds.

We begin with the case when $A$ is SPD.
Since the Schur complement $S$ has the form~\eqref{B_aux}, \cref{Thm:condition_number_aux} yields the following identities.

\begin{theorem}
\label{Thm:Schur}
In the saddle point problem~\eqref{model_saddle}, suppose that $A$ is SPD.
Then the Schur complement $S$ given in~\eqref{Schur} satisfies
\begin{equation*}
    \lambda_{\min} (S)
    = \inf_{0 \neq q \in W} \sup_{v \in V,\ Bv = q} \frac{\| q \|^2}{(Av, v)}, \quad
    \lambda_{\max} (S)
    = \sup_{0 \neq v \in V} \frac{\| Bv \|^2}{(Av, v)}.
\end{equation*}
\end{theorem}
\begin{proof}
In \cref{Thm:condition_number_aux}, we set
\begin{equation*}
    V \leftarrow W, \quad \undertilde{V} \leftarrow V, \quad
    L \leftarrow I, \quad M \leftarrow S, \quad
    \Pi \leftarrow B, \quad \undertilde{M} \leftarrow A^{-1}.
\end{equation*}
Then we get
\begin{equation*}
    \lambda_{\min} (S)
    = \left( \sup_{0 \neq q \in W} \inf_{v \in V,\ Bv = q} \frac{(Av, v)}{\| q \|^2} \right)^{-1}
    = \inf_{0 \neq q \in W} \sup_{v \in V,\ Bv = q} \frac{\| q \|^2}{(Av, v)}.
\end{equation*}
Moreover, we have
\begin{equation*}
    \lambda_{\max} (S)
    = \left( \inf_{0 \neq q \in W} \inf_{v \in V,\ Bv = q} \frac{(Av, v)}{\| q \|^2} \right)^{-1}
    = \left( \inf_{v \in V,\ Bv \neq 0} \frac{(Av, v)}{\| Bv \|^2 } \right)^{-1}
    = \sup_{0 \neq v \in V} \frac{\| Bv \|^2}{(Av, v)},
\end{equation*}
which completes the proof.
\end{proof}

\begin{remark}
\label{Rem:BBF1}
We note that the minimum eigenvalue estimate in \cref{Thm:Schur} can also be derived from some well-known properties of saddle point systems presented in~\cite{BBF:2013}.
Invoking~\cite[Proposition~3.4.3]{BBF:2013} for $BA^{-\frac{1}{2}}$ yields
\begin{equation}
\label{Rem1:BBF1}
\|(BA^{-\frac{1}{2}})^{-1}\|^{-1}
= \inf_{0 \neq q \in W} \sup_{0 \neq w \in \mathcal{N}(BA^{-\frac{1}{2}})^\perp}
    \frac{(BA^{-\frac{1}{2}}w,q)}{\|w\|\|q\|}
= \inf_{0 \neq q \in W} \frac{\|(BA^{-\frac{1}{2}})^tq\|}{\|q\|},
\end{equation}
where $BA^{-\frac{1}{2}}$ is regarded as an isomorphism from $\mathcal{N}(BA^{-\frac{1}{2}})^\perp$ onto $W$.
The second equality follows from the Cauchy--Schwarz inequality, with equality attained at $w=(BA^{-\frac{1}{2}})^tq \in \mathcal{N}(BA^{-\frac{1}{2}})^\perp$.
Since $S = (BA^{-\frac{1}{2}})(BA^{-\frac{1}{2}})^t$, we get
\begin{equation}
\label{Rem2:BBF1}
\|(BA^{-\frac{1}{2}})^{-1}\|^{-2}
\stackrel{\eqref{Rem1:BBF1}}{=} \inf_{0 \neq q \in W} \frac{\|(BA^{-\frac{1}{2}})^tq\|^2}{\|q\|^2}
= \inf_{0 \neq q \in W} \frac{(Sq,q)}{\|q\|^2}
= \lambda_{\min}(S).
\end{equation}
It remains to express $\|(BA^{-\frac{1}{2}})^{-1}\|$ in terms of the constraint $Bv=q$.
If $BA^{-\frac{1}{2}}w=q$, then $w-(BA^{-\frac{1}{2}})^{-1}q \in \mathcal{N}(BA^{-\frac{1}{2}})$, whereas $(BA^{-\frac{1}{2}})^{-1}q \in \mathcal{N}(BA^{-\frac{1}{2}})^\perp$.
Hence,
\begin{equation}
\label{Rem3:BBF1}
\|w\|^2
= \|(BA^{-\frac{1}{2}})^{-1}q\|^2 + \|w-(BA^{-\frac{1}{2}})^{-1}q\|^2
\geq \|(BA^{-\frac{1}{2}})^{-1}q\|^2,
\end{equation}
with equality at $w=(BA^{-\frac{1}{2}})^{-1}q$.
Consequently,
\begin{multline*}
\lambda_{\min}(S)^{-1}
\stackrel{\eqref{Rem2:BBF1}}{=} \|(BA^{-\frac{1}{2}})^{-1}\|^2
= \sup_{0 \neq q \in W} \frac{\|(BA^{-\frac{1}{2}})^{-1}q\|^2}{\|q\|^2} \\
\stackrel{\eqref{Rem3:BBF1}}{=} \sup_{0 \neq q \in W} \inf_{w \in V,\ BA^{-\frac{1}{2}}w=q}
    \frac{\|w\|^2}{\|q\|^2}
= \sup_{0 \neq q \in W} \inf_{v \in V,\ Bv=q}
    \frac{(Av,v)}{\|q\|^2},
\end{multline*}
where the last equality uses the change of variables $w=A^{\frac{1}{2}}v$.
Taking reciprocals yields the first identity in \cref{Thm:Schur}.
A similar argument was used, for example, in~\cite[Lemma~2]{MS:2007}.
\end{remark}

\begin{remark}
\label{Rem:BBF2}
The maximum eigenvalue estimate in \cref{Thm:Schur} is easily obtained from $S = (BA^{-\frac{1}{2}})(BA^{-\frac{1}{2}})^t$.
Namely, we have
\begin{equation*}
\lambda_{\max}(S)
= \|BA^{-\frac{1}{2}}\|^2
= \sup_{0 \neq w \in V} \frac{\|BA^{-\frac{1}{2}}w\|^2}{\|w\|^2}
= \sup_{0 \neq v \in V} \frac{\|Bv\|^2}{(Av,v)},
\end{equation*}
where we use the change of variables $w=A^{\frac{1}{2}}v$.
\end{remark}

As a corollary of \cref{Thm:Schur}, we obtain \cref{Cor:Schur}, which appeared in~\cite[Lemma~9.1]{TW:2005} (see also~\cite{SW:1994}).

\begin{corollary}
\label{Cor:Schur}
In the saddle point problem~\eqref{model_saddle}, suppose that $A$ is SPD.
Then the Schur complement $S$ given in~\eqref{Schur} satisfies
\begin{equation*}
    \lambda_{\min} (S)
    \geq \lambda_{\max} (A)^{-1} \| B^{-1} \|^{-2}, \quad
    \lambda_{\max} (S)
    \leq \lambda_{\min} (A)^{-1} \| B \|^2.
\end{equation*}
\end{corollary}
\begin{proof}
By \cref{Thm:Schur}, we obtain
\begin{multline*}
    \lambda_{\min} (S) = \inf_{0 \neq q \in W} \sup_{v \in V,\ Bv = q} \frac{\| q \|^2}{(Av, v)}
    \geq \lambda_{\max} (A)^{-1} \inf_{0 \neq q \in W} \sup_{v \in V,\ Bv = q} \frac{\| q \|^2}{\| v \|^2} \\
    = \lambda_{\max} (A)^{-1} \left( \sup_{0 \neq q \in W} \inf_{v \in V,\ Bv = q} \frac{\| v \|^2}{\| q \|^2} \right)^{-1}
    = \lambda_{\max} (A)^{-1} \| B^{-1} \|^{-2},
\end{multline*}
where in the last equality we use the fact that the minimizer $v \in V$ of $\|v\|^2/\|q\|^2$ subject to $Bv=q$ must lie in $\mathcal{N}(B)^{\perp}$.
Meanwhile, we have
\begin{equation*}
    \lambda_{\max} (S) = \sup_{0 \neq v \in V} \frac{\| Bv \|^2}{(Av, v)}
    \leq \lambda_{\min} (A)^{-1} \sup_{0 \neq v \in V} \frac{\| Bv \|^2}{\| v \|^2} = \lambda_{\min} (A)^{-1} \| B \|^2,
\end{equation*}
which completes the proof.
\end{proof}

\begin{remark}
\label{Rem:Schur_insight}
\cref{Thm:Schur} provides useful insight into the spectral analysis of Schur complements.
For the minimum eigenvalue, we seek a larger lower bound.
The supremum over $v$ satisfying $Bv=q$ in the first identity suggests choosing $v$ with small energy $(Av,v)$ relative to $\|q\|^2$.
Thus, finding a suitable right inverse of $B$ that is stable in the $A$-norm is important.
For the maximum eigenvalue, we seek a smaller upper bound.
Since $Bv=B(v+\phi)$ for any $\phi\in\mathcal{N}(B)$, we may estimate $\|B(v+\phi)\|^2$ in the numerator while keeping the original energy $(Av,v)$ in the denominator.
Choosing $\phi$ so that $v+\phi$ is small in a suitable norm can therefore yield a sharper estimate.
Both arguments use the interaction between the constraint operator $B$ and the energy induced by $A$, whereas \cref{Cor:Schur} estimates their effects separately.
These observations are used in the applications in \cref{Sec:Mixed,Sec:FETI}.
\end{remark}

Sometimes a preconditioner $L \colon W \to W$ is applied to the Schur complement $S$. 
\cref{Thm:Schur_preconditioned} provides sharp estimates for the extremal eigenvalues of the preconditioned operator $LS$.

\begin{theorem}
\label{Thm:Schur_preconditioned}
In the saddle point problem~\eqref{model_saddle}, suppose that $A$ is SPD.
Let $L \colon W \to W$ be an SPD linear operator.
Then the Schur complement $S$ given in~\eqref{Schur} satisfies
\begin{equation*}
    \lambda_{\min} (L S)
    = \inf_{0 \neq q \in W} \sup_{v \in V,\ Bv = q} \frac{(Lq, q)}{(Av, v)}, \quad
    \lambda_{\max} (L S)
    = \sup_{0 \neq v \in V} \frac{(LBv, Bv)}{(Av, v)}.
\end{equation*}
\end{theorem}
\begin{proof}
Setting
\begin{equation*}
    V \leftarrow W, \quad \undertilde{V} \leftarrow V, \quad
    L \leftarrow L, \quad M \leftarrow S, \quad
    \Pi \leftarrow B, \quad \undertilde{M} \leftarrow A^{-1}
\end{equation*}
in \cref{Thm:condition_number_aux} yields the desired result.
\end{proof}

An interesting special case arises when a right inverse of $B$ is explicitly available. 
In this situation, one can construct a special preconditioner $L$ using the right inverse of $B$ as a building block; examples include certain nonoverlapping domain decomposition methods, which we will discuss in \cref{Sec:FETI}. 
The following corollary gives a direct variational proof of the spectral bounds in~\cite[Lemma~2]{MS:2007}.

\begin{corollary}
\label{Cor:Schur_preconditioned}
In the saddle point problem~\eqref{model_saddle}, suppose that $A$ is SPD.
Assume further that a linear operator $\bar{B} \colon V \to W$ satisfies $B \bar{B}^t = I$, and define a preconditioner $L \colon W \to W$ by
\begin{equation*}
    L = \bar{B} A \bar{B}^t.
\end{equation*}
Then the Schur complement $S$ given in~\eqref{Schur} satisfies
\begin{equation*}
    \lambda_{\min} (L S) \geq 1, \quad \lambda_{\max} (L S)
    = \sup_{0 \neq v \in V} \frac{(A \bar{B}^t B v, \bar{B}^t B v)}{(Av, v)}
    = \| \bar{B}^t B \|_A^2,
\end{equation*}
where $\| \cdot \|_A$ denotes the operator norm with respect to the $A$-norm on $V$.
\end{corollary}
\begin{proof}
The estimate for the minimum eigenvalue follows by substituting $v = \bar{B}^t q$ into the expression in \cref{Thm:Schur_preconditioned}, 
while the estimate for the maximum eigenvalue is an immediate consequence of \cref{Thm:Schur_preconditioned}.
\end{proof}

Next, we consider the case when $A$ is semi-SPD. 
Recall that the projected Schur complement $S_0$ given in~\eqref{Schur_singular} satisfies
\begin{equation*}
    S_0 = PB A^+ (PB)^t.
\end{equation*}
As discussed in the proof of \cref{Prop:Schur_singular}, the operator $PB|_{\mathcal{R}(A)}$ is surjective and $\mathcal{R}((PB)^t) \subset \mathcal{R}(A)$. 
Hence, in view of \cref{Prop:Schur_singular}, we may choose $A^+$ to coincide with the inverse of $A$ on $\mathcal{R}(A)$ and regard $A^+$ and $PB$ as
\begin{equation}
\label{domain_restriction}
    A^+ \colon \mathcal{R}(A) \to \mathcal{R}(A), 
    \quad
    PB \colon \mathcal{R}(A) \to W_0,
\end{equation}
where $W_0$ was defined in~\eqref{W0}.
This observation shows that, with the domain restrictions in~\eqref{domain_restriction}, the semi-SPD case reduces directly to the SPD case, and the same argument yields sharp spectral estimates; see \cref{Thm:Schur_singular,Thm:Schur_singular_preconditioned}.
Consequently, \cref{Thm:Schur_singular,Thm:Schur_singular_preconditioned} do not depend on a particular choice of the pseudoinverse $A^+$ on the whole space $V$, but only on the canonical inverse of $A$ on $\mathcal{R}(A)$.
In the following spectral estimates, we assume that $W_0 \neq \{0\}$.

\begin{theorem}
\label{Thm:Schur_singular}
In the saddle point problem~\eqref{model_saddle}, suppose that $A$ is semi-SPD.
Then the projected Schur complement $S_0$ given in~\eqref{Schur_singular} satisfies
\begin{equation*}
    \lambda_{\min} (S_0)
    = \inf_{0 \neq q \in W_0} \sup_{v \in \mathcal{R}(A),\ PBv = q} \frac{\| q \|^2}{(Av, v)}, \quad
    \lambda_{\max} (S_0)
    = \sup_{0 \neq v \in \mathcal{R} (A)} \frac{\| PBv \|^2}{(Av, v)}.
\end{equation*}
\end{theorem}

\begin{theorem}
\label{Thm:Schur_singular_preconditioned}
In the saddle point problem~\eqref{model_saddle}, suppose that $A$ is semi-SPD.
Let $L \colon W_0 \to W_0$ be an SPD linear operator.
Then the projected Schur complement $S_0$ given in~\eqref{Schur_singular} satisfies
\begin{equation*}
    \lambda_{\min} (L S_0)
    = \inf_{0 \neq q \in W_0} \sup_{v \in \mathcal{R}(A),\ PBv = q} \frac{(Lq, q)}{(Av, v)}, \quad
    \lambda_{\max} (L S_0)
    = \sup_{0 \neq v \in \mathcal{R} (A)} \frac{(LPBv, PBv)}{(Av, v)}.
\end{equation*}
\end{theorem}
\begin{proof}
Setting
\begin{equation*}
    V \leftarrow W_0, \quad \undertilde{V} \leftarrow \mathcal{R}(A), \quad
    L \leftarrow L, \quad M \leftarrow S_0, \quad
    \Pi \leftarrow PB, \quad \undertilde{M} \leftarrow A^+
\end{equation*}
in \cref{Thm:condition_number_aux} yields the desired result.
\end{proof}

For completeness, we also provide a counterpart of \cref{Cor:Schur_preconditioned} for the semi-SPD case in \cref{Cor:Schur_singular_preconditioned}.

\begin{corollary}
\label{Cor:Schur_singular_preconditioned}
In the saddle point problem~\eqref{model_saddle}, suppose that $A$ is semi-SPD.
Assume further that a linear operator $\bar{B} \colon V \to W$ satisfies $B \bar{B}^t = I$, and define a preconditioner $L \colon W_0 \to W_0$ by
\begin{equation*}
    L = (P \bar{B}) A (P \bar{B})^t.
\end{equation*}
Then the projected Schur complement $S_0$ given in~\eqref{Schur_singular} satisfies
\begin{equation*}
    \lambda_{\min} (L S_0) \geq 1, \quad
    \lambda_{\max} (L S_0)
    = \sup_{0 \neq v \in \mathcal{R} (A)} \frac{(A (P \bar{B})^t P B v, (P \bar{B})^t P B v)}{(Av, v)}.
\end{equation*}
\end{corollary}

\section{Application I: Augmented Lagrangian method}
\label{Sec:ALM} 
In the remainder of this paper, we illustrate the use of the Schur complement estimates in \cref{Sec:Schur} through several applications.
As a first example, we consider the augmented Lagrangian method~\cite{FG:1983}, which is a stationary iterative method for solving the saddle point problem~\eqref{model_saddle}.

In the augmented Lagrangian method, given the current iterate $(u^{m}, p^{m})$, the next iterate $(u^{m+1}, p^{m+1})$ is obtained as
\begin{equation}
\label{ALM}
\begin{aligned}
u^{m+1} &= (A + \epsilon^{-1} B^t B)^{-1} (f + \epsilon^{-1} B^t g - B^t p^{m}), \\
p^{m+1} &= p^{m} - \epsilon^{-1}(g - Bu^{m+1}),
\end{aligned}
\quad m \geq 0.
\end{equation}
Here, $\epsilon>0$ is the augmented Lagrangian parameter, and $A+\epsilon^{-1}B^tB$ is SPD by~\eqref{well-posed}, even when $A$ is only semi-SPD.
It is well-known (see, e.g.,~\cite{BGL:2005}) that the augmented Lagrangian method~\eqref{ALM} is equivalent to the Richardson iteration with step size $\epsilon^{-1}$ applied to the linear system
\begin{equation}
\label{Schur_augmented}
S_{\epsilon} p = d_{\epsilon},
\text{ where }
S_{\epsilon} = B (A + \epsilon^{-1} B^t B)^{-1} B^t, \,
d_{\epsilon} = B (A + \epsilon^{-1} B^t B)^{-1} (f + \epsilon^{-1} B^t g) - g.
\end{equation}
Therefore, the analysis of the augmented Lagrangian method reduces to the analysis of the augmented Schur complement $S_{\epsilon}$.

We first give a short variational derivation of the relation between $S_{\epsilon}$ and $S$.
Assume first that $A$ is SPD.
By repeated applications of \cref{Lem:aux_space_lemma}, for any $q \in W$ we have
\begin{multline*}
(S_{\epsilon}^{-1} q, q)
= \inf_{v \in V,\ Bv = q} ((A + \epsilon^{-1}B^t B)v, v) \\
= \inf_{v \in V,\ Bv = q} (Av, v) + \epsilon^{-1}(q,q) 
= (S^{-1}q, q) + \epsilon^{-1}(q,q).
\end{multline*}
Thus we obtain the identity (cf.~\cite[equation~(3.4)]{LP:2017}):
\begin{equation}
\label{S_epsilon}
S_{\epsilon}^{-1} = S^{-1} + \epsilon^{-1} I.
\end{equation}
Consequently, the extremal eigenvalues of $S_{\epsilon}$ are
\begin{equation}
\label{Schur_augmented_eigenvalues}
\lambda_{\min}(S_{\epsilon}) = \frac{\epsilon \lambda_{\min}(S)}{\epsilon + \lambda_{\min}(S)}, \quad
\lambda_{\max}(S_{\epsilon}) = \frac{\epsilon \lambda_{\max}(S)}{\epsilon + \lambda_{\max}(S)}.
\end{equation}

In the augmented Lagrangian method~\eqref{ALM}, the error propagation operator is given by $I  - \epsilon^{-1} S_{\epsilon}$.
Using~\eqref{Schur_augmented_eigenvalues}, we have the following sharp estimate:
\begin{equation}
\label{ALM_sharp}
\| I - \epsilon^{-1} S_{\epsilon} \|
= \lambda_{\max} ( I - \epsilon^{-1} S_{\epsilon})
= 1 - \epsilon^{-1} \lambda_{\min} (S_{\epsilon})
= \frac{\epsilon}{\epsilon + \lambda_{\min} (S)}.
\end{equation}
Namely, the convergence rate of~\eqref{ALM} becomes arbitrarily fast as $\epsilon$ becomes small.
The estimate follows directly from~\eqref{S_epsilon}; related estimates are given in~\cite[Lemma~2.1]{LWXZ:2007}.
We note that such arbitrarily fast convergence of the augmented Lagrangian method can be extended to a broader class of convex optimization problems; see~\cite[Appendix~A]{PH:2025}.

On the other hand, if we apply the conjugate gradient method to solve~\eqref{Schur_augmented}, then the convergence rate depends on the condition number $\kappa(S_{\epsilon})$, which satisfies
\begin{equation*}
\kappa(S_{\epsilon}) 
= \frac{\lambda_{\max}(S)}{\lambda_{\min}(S)} 
  \frac{\epsilon + \lambda_{\min}(S)}{\epsilon + \lambda_{\max}(S)}
\rightarrow 1 
\quad \text{as} \quad \epsilon \to 0.
\end{equation*}
This behavior agrees with~\cite{LP:2009,LP:2017}, where the conditioning of augmented Lagrangian formulations was analyzed for nonoverlapping domain decomposition methods.

\begin{remark}
\label{Rem:S_epsilon}
Alternative derivations of the identity~\eqref{S_epsilon} can be found in the literature.
For example, the well-known Sherman--Morrison--Woodbury formula~\cite[equation~(2.1.4)]{GV:2013},
\begin{equation*}
    (A + UCV)^{-1} = A^{-1} - A^{-1} U (C^{-1} + V A^{-1} U)^{-1} V A^{-1},
\end{equation*}
where $A$, $C$, $U$, and $V$ are conformable matrices, can be used to derive~\eqref{S_epsilon}; see, for example,~\cite[Lemma~4.1]{BO:2006} and~\cite[equation~(3.4)]{LP:2017}.
\end{remark}

Next, we consider the case when $A$ is semi-SPD.
Proceeding similarly to the SPD case, we obtain the following counterpart of~\eqref{S_epsilon}:
\begin{equation*}
S_{\epsilon}^{-1}=P^tS_0^{-1}P+\epsilon^{-1}I.
\end{equation*}
The corresponding error propagation operator is given by
\begin{equation}
\label{ALM_error_singular}
I-\epsilon^{-1}S_{\epsilon}
=\epsilon P^t(\epsilon I+S_0)^{-1}P.
\end{equation}
Consequently, for any initial multiplier $p^0$, the iterates of the augmented Lagrangian method~\eqref{ALM} satisfy
\begin{equation*}
(I-P^tP)(p^m-p)=0,
\quad m \geq 1.
\end{equation*}
Namely, the augmented Lagrangian method eliminates the component of the multiplier error in $W_0^{\perp}$ in one iteration.
These are precisely the multiplier directions associated with $\mathcal{N}(A)$ that are removed by the projection in~\eqref{Schur_singular}.
Thus, after one iteration, the remaining multiplier error evolves on $W_0$ through the operator $\epsilon(\epsilon I+S_0)^{-1}$, which has the same form as in the SPD case.
In particular, if $W_0 \neq \{0\}$, then~\eqref{ALM_error_singular} yields
\begin{equation*}
\|I-\epsilon^{-1}S_{\epsilon}\|
=\frac{\epsilon}{\epsilon+\lambda_{\min}(S_0)},
\end{equation*}
which is the same sharp estimate as the SPD case~\eqref{ALM_sharp} with $S$ replaced by $S_0$.

\section{Application II: Mixed finite element methods}
\label{Sec:Mixed}
In this section, we present an analysis of several mixed finite element methods for solving mixed variational problems~\cite{BBF:2013}. 
As representative examples, we consider Darcy flow~\cite{BOS:2018,PH:2025} and the Stokes equations~\cite{ABMXZ:2014,Verfurth:1984}, which have broad applications in fluid mechanics.

\subsection{Darcy flow}
Let $\Omega$ be a bounded polyhedral domain in $\mathbb{R}^d$ ($d=2,3$). 
The linear relationship between the Darcy velocity $u$ and pressure $p$, together with the conservation of mass and the homogeneous Dirichlet boundary condition, is given by
\begin{equation}
\label{Darcy_strong}
\begin{aligned}
u + \nabla p &= 0 \quad \text{in } \Omega, \\
\operatorname{div} u &= b \quad \text{in } \Omega, \\
p &= 0 \quad \text{on } \partial \Omega,
\end{aligned}
\end{equation}
where $b$ is a prescribed source or sink term. 
Problems of the same form also arise in the mixed formulation of the Poisson equation, as well as in the Hellinger--Reissner variational principle for linear elasticity.

The weak formulation of~\eqref{Darcy_strong} defined on $H(\operatorname{div}; \Omega) \times L^2(\Omega)$ reads as follows:  
find $u \in H(\operatorname{div}; \Omega)$ and $p \in L^2(\Omega)$ such that
\begin{equation}
\label{Darcy_weak}
\begin{aligned}
\int_{\Omega} u \cdot v \,dx - \int_{\Omega} p \,\operatorname{div} v \,dx &= 0, \\
- \int_{\Omega} q \,\operatorname{div} u \,dx &= - \int_{\Omega} b q \,dx,
\end{aligned}
\quad v \in H(\operatorname{div}; \Omega),\ q \in L^2(\Omega).
\end{equation}
A mixed finite element method for solving~\eqref{Darcy_weak} is obtained by replacing $H(\operatorname{div}; \Omega)$ and $L^2(\Omega)$ with suitable finite element spaces $V$ and $W$.
We assume that the finite element meshes are shape-regular and quasi-uniform.
The corresponding algebraic system is
\begin{equation}
\label{Darcy_FEM}
\begin{bmatrix} M_V & B^t \\ B & 0 \end{bmatrix} 
\begin{bmatrix} u \\ p \end{bmatrix}
= \begin{bmatrix} 0 \\ g \end{bmatrix},
\end{equation}
where the matrices $M_V$, $B$, and the vector $g$ are defined by
\begin{equation*}
\resizebox{\textwidth}{!}{$\displaystyle
(M_V v, w) = \int_{\Omega} v \cdot w \,dx,\
(Bv, q) = - \int_{\Omega} q \,\operatorname{div} v \,dx,\
(g, q) = - \int_{\Omega} b q \,dx,
\quad v, w \in V,\ q \in W,
$}
\end{equation*}
and $(\cdot,\cdot)$ denotes the Euclidean inner product. 
Thus, the problem~\eqref{Darcy_FEM} is an instance of~\eqref{model_saddle}.

It is straightforward to verify that
\begin{equation}
\label{discrete_divergence}
    B = - M_W \operatorname{div},
\end{equation}
where $M_W \colon W \to W$ is the mass matrix on $W$ defined analogously to $M_V$. 
Moreover, the mass matrix $M_W$ satisfies the scaling relation~(cf.~\cite[Section~3.4]{TW:2005})
\begin{equation}
\label{mass_matrix}
    (M_W q, q) \eqsim h^d \| q \|^2, 
    \quad q \in W,
\end{equation}
where $h > 0$ denotes the characteristic element diameter.

Following \cref{Rem:Schur_insight}, we seek a right inverse of $B$ with controlled energy $(M_Vv,v)=\|v\|_{L^2}^2$.
In view of~\eqref{discrete_divergence}, such a right inverse can be obtained from the discrete Babu\v{s}ka--Brezzi condition.
We assume that the finite element spaces $V$ and $W$ satisfy this condition (see, e.g.,~\cite{BOS:2018}), namely,
\begin{equation*}
    \inf_{0 \neq q \in W} \sup_{0 \neq v \in V} \frac{(\operatorname{div} v, q)_{L^2}}{\| v \|_{H(\operatorname{div})} \| q \|_{L^2}} \gtrsim 1.
\end{equation*}
By~\eqref{Brezzi}, this condition is equivalent to the statement that for any $q \in W$, there exists $v \in V$ such that
\begin{equation}
    \label{Brezzi_Darcy}
    \operatorname{div} v = q, 
    \quad \| q \|_{L^2} \gtrsim \| v \|_{H(\operatorname{div})}.
\end{equation}

Now, we analyze the Schur complement $S = B M_V^{-1} B^t$ corresponding to~\eqref{Darcy_FEM}.
By \cref{Thm:Schur}, the minimum eigenvalue of $S$ is estimated as follows:
\begin{multline*}
\lambda_{\min} (S) = \inf_{0 \neq q \in W} \sup_{v \in V,\ Bv = q} \frac{\| q \|^2}{(M_V v, v)}
\stackrel{\eqref{discrete_divergence}}{=} \inf_{0 \neq q \in W} \sup_{v \in V, \operatorname{div} v = q} \frac{\| M_W q \|^2}{(M_V v, v)} \\
\stackrel{\eqref{mass_matrix}}{\eqsim} h^{d} \inf_{0 \neq q \in W} \sup_{v \in V, \operatorname{div} v = q} \frac{(M_W q, q)}{(M_V v, v)}
\stackrel{\eqref{Darcy_FEM}}{=} h^{d} \inf_{0 \neq q \in W} \sup_{v \in V, \operatorname{div} v = q} \frac{\| q \|_{L^2}^2}{\| v \|_{L^2}^2}
\stackrel{\eqref{Brezzi_Darcy}}{\gtrsim} h^{d}.
\end{multline*}
For the maximum eigenvalue, we have
\begin{equation*}
\lambda_{\max} (S)
= \sup_{0 \neq v \in V} \frac{\| B v \|^2}{(M_V v, v)}
\stackrel{\eqref{discrete_divergence}}{=} \sup_{0 \neq v \in V} \frac{\| M_W \operatorname{div} v \|^2}{(M_V v, v)}
\stackrel{\eqref{mass_matrix}}{\eqsim} h^d \sup_{0 \neq v \in V} \frac{\| \operatorname{div} v \|_{L^2}^2}{\| v \|_{L^2}^2}
\lesssim h^{d-2},
\end{equation*}
where the last inequality follows from the standard inverse inequality~(see, e.g.,~\cite[Lemma~2.2]{Oh:2013}).
In conclusion, we obtain
\begin{equation*}
    \kappa (S) \lesssim h^{-2}.
\end{equation*}

The bound on $\kappa(S)$ deteriorates as the element diameter $h$ decreases, motivating the use of a preconditioner. 
As discussed in~\cite{Xu:2010}, an optimal preconditioner for $S$ can be constructed by exploiting its spectral equivalence with a certain discretization of the Poisson problem. 
Indeed, a preconditioner based on the spectral equivalence with the Poisson problem discretized by an interior penalty method was considered in~\cite{RVW:1996}.

\subsection{Stokes equations}
Let $\Omega$ be a bounded polyhedral domain in $\mathbb{R}^d$ ($d=2,3$). 
The incompressible Stokes equations with the homogeneous Dirichlet boundary condition are written as
\begin{equation}
\label{Stokes_strong}
\begin{aligned}
- \Delta u + \nabla p &= f \quad &&\text{in } \Omega, \\
\operatorname{div} u &= 0 \quad &&\text{in } \Omega, \\
u &= 0 \quad &&\text{on } \partial \Omega,
\end{aligned}
\end{equation}
where $u$ denotes the velocity, $p$ the pressure, and $f$ the external body force density.

The weak formulation of~\eqref{Stokes_strong} on $H_0^1(\Omega)^d$ and $L_0^2(\Omega)$ is given as follows: find $(u,p) \in H_0^1(\Omega)^d \times L_0^2(\Omega)$ such that
\begin{equation*}
\begin{aligned}
\int_{\Omega} \nabla u \cdot \nabla v \,dx - \int_{\Omega} p \,\operatorname{div} v \,dx &= \int_{\Omega} f \cdot v \,dx, \\
- \int_{\Omega} q \,\operatorname{div} u \,dx &= 0,
\end{aligned}
\quad v \in H_0^1(\Omega)^d,\ q \in L_0^2(\Omega).
\end{equation*}
A mixed finite element method for solving~\eqref{Stokes_strong} is obtained by replacing $H_0^1(\Omega)^d$ and $L_0^2(\Omega)$ with suitable finite element spaces $V$ and $W$.
Then the corresponding algebraic system becomes~\eqref{model_saddle} with $g = 0$, where the matrices $A$, $B$ and the vector $f$ are defined by
\begin{equation}
\label{Stokes_FEM}
\resizebox{\textwidth}{!}{$\displaystyle
(A v, w) = \int_{\Omega} \nabla v \cdot \nabla w\,dx,\
(Bv, q) = - \int_{\Omega} q \operatorname{div} v \,dx,\ 
(f, v) = \int_{\Omega} f \cdot v \,dx, \quad
v,w \in V,\ q \in W. 
$}
\end{equation}

Let $M_W \colon W \to W$ be the mass matrix on $W$. 
Then, similarly to the case of Darcy flow, we observe that~\eqref{discrete_divergence} and~\eqref{mass_matrix} hold. 
Following \cref{Rem:Schur_insight}, we again seek a right inverse of $B$, now with controlled energy $(Av,v)=|v|_{H^1}^2$.
In view of~\eqref{discrete_divergence}, such a right inverse is provided by the discrete Babu\v{s}ka--Brezzi condition.
We assume that the finite element spaces $V$ and $W$ satisfy this condition (see~\cite{BBF:2013}), namely,
\begin{equation*}
\inf_{0 \neq q \in W} \sup_{0 \neq v \in V} \frac{(\operatorname{div} v, q)_{L^2}}{\| v \|_{H^1} \|q \|_{L^2}} \gtrsim 1.
\end{equation*}
By~\eqref{Brezzi}, this condition is equivalent to the statement that for any $q \in W$, there exists $v \in V$ such that
\begin{equation}
\label{Brezzi_Stokes}
\operatorname{div} v = q, \quad \| q \|_{L^2} \gtrsim \| v \|_{H^1}.
\end{equation}

Now, we estimate the extremal eigenvalues of the Schur complement $S = BA^{-1} B^t$.
By \cref{Thm:Schur}, the minimum eigenvalue of $S$ is estimated as follows:
\begin{multline*}
    \lambda_{\min} (S)
    = \inf_{0 \neq q \in W} \sup_{v \in V,\ Bv = q} \frac{\| q \|^2}{(Av, v)}
    \stackrel{\eqref{discrete_divergence}}{=} \inf_{0 \neq q \in W} \sup_{v \in V,\ \operatorname{div} v = q} \frac{\| M_W q \|^2}{(Av, v)} \\
    \stackrel{\eqref{mass_matrix}}{\eqsim} h^{d} \inf_{0 \neq q \in W} \sup_{v \in V,\ \operatorname{div} v = q} \frac{(M_W q, q)}{(Av, v)}
    \stackrel{\eqref{Stokes_FEM}}{=} h^{d} \inf_{0 \neq q \in W} \sup_{v \in V,\ \operatorname{div} v = q} \frac{\| q \|_{L^2}^2}{| v |_{H^1}^2} \stackrel{\eqref{Brezzi_Stokes}}{\gtrsim} h^{d}.
\end{multline*}
On the other hand, the maximum eigenvalue of $S$ is estimated as follows:
\begin{multline*}
    \lambda_{\max} (S)
    = \sup_{0 \neq v \in V} \frac{\| Bv \|^2}{(Av, v)}
    \stackrel{\eqref{discrete_divergence}}{=} \sup_{0 \neq v \in V} \frac{\| M_W \operatorname{div} v \|^2}{(Av, v)} \\
    \stackrel{\eqref{mass_matrix}}{\eqsim} h^d \sup_{0 \neq v \in V} \frac{(M_W \operatorname{div} v, \operatorname{div} v)}{(Av, v)}
    \stackrel{\eqref{Stokes_FEM}}{=} h^d \sup_{0 \neq v \in V} \frac{\| \operatorname{div} v \|_{L^2}^2}{| v |_{H^1}^2}
    \lesssim h^d.
\end{multline*}
Consequently, we obtain
\begin{equation*}
    \kappa (S) \lesssim 1.
\end{equation*}
This implies that the conjugate gradient method for solving the dual problem converges uniformly with respect to the characteristic element diameter $h$, which agrees with~\cite{Verfurth:1984}.

\section{Application III: FETI and FETI-DP}
\label{Sec:FETI}
In this section, we apply the estimates in \cref{Sec:Schur} to two widely used nonoverlapping domain decomposition methods~\cite{TW:2005,XZ:1998}: FETI~\cite{FR:1991} and FETI-DP~\cite{FLLPR:2001,KWD:2002}.
We recover condition number bounds for their unpreconditioned and preconditioned formulations, emphasizing the roles of right inverses and null space corrections.

\subsection{FETI}
We first consider FETI, originally proposed in~\cite{FR:1991}. 
In FETI, a global problem is partitioned into smaller subproblems posed on subdomains, and continuity across subdomain interfaces is enforced by means of Lagrange multipliers.

To illustrate, let $\Omega$ be a bounded polygonal domain in $\mathbb{R}^2$, and consider the model Poisson equation
\begin{equation}
\label{Poisson}
\begin{aligned}
- \Delta u &= f \quad \text{in } \Omega, \\
u &= 0 \quad \text{on } \partial \Omega.
\end{aligned}
\end{equation}
Assume that the domain $\Omega$ is decomposed into $J$ nonoverlapping polygonal subdomains $\{ \Omega_j \}_{j=1}^J$ with characteristic subdomain diameter $H>0$, where each $\Omega_j$ is a union of coarse elements from a global conforming mesh $\mathcal{T}_H$. 
We further assume that each coarse element in $\mathcal{T}_H$ is itself a union of fine elements from a fine mesh $\mathcal{T}_h$. 
For $i<j$, let $\Gamma_{ij} = \partial \Omega_i \cap \partial \Omega_j$ denote the interface between adjacent subdomains, and define $\Gamma = \bigcup_{i<j} \Gamma_{ij}$.
We also assume that the fine mesh is quasi-uniform and that the interface edge lengths are comparable to $H$.
If $\partial\Omega_j\cap\partial\Omega\neq\emptyset$, we assume that this intersection contains an edge of length comparable to $H$.

In what follows, the indices $i<j$ are taken over $1,\dots,J$. 
Let $\mathcal{T}_h^j$ denote the restriction of the global triangulation $\mathcal{T}_h$ to $\Omega_j$. 
Then $\mathcal{T}_h^i$ and $\mathcal{T}_h^j$ share nodal points along the interface $\Gamma_{ij}$. 
On each $\Omega_j$, we consider the space of continuous, piecewise linear finite elements on $\mathcal{T}_h^j$ that vanish on $\partial \Omega_j \cap \partial \Omega$, and denote its restriction to the interface $\partial \Omega_j$ by $V_j$. 
We also define the product space $V = \prod_{j=1}^J V_j$. 
Note that functions in $V$ are, in general, discontinuous across $\Gamma$.

The problem of interest is
\begin{equation}
\label{FETI_constrained}
\min_{v = (v_j)_{j=1}^J \in V} \left\{ \frac{1}{2} (Av, v) - (f, v) \right\}
\quad \text{subject to} \quad Bv = 0,
\end{equation}
where the matrix $A$ and the vector $f$ are defined by
\begin{equation}
\label{discrete_harmonic}
(Av, w) = \sum_{j=1}^J \int_{\Omega_j} \nabla \mathcal{H}_j v_j \cdot \nabla \mathcal{H}_j w_j \,dx, 
\quad
(f, v) = \sum_{j=1}^J \int_{\Omega_j} f \,\mathcal{H}_j v_j \,dx,
\quad v, w \in V,
\end{equation}
where $\mathcal{H}_j$ denotes the discrete harmonic extension in $\Omega_j$ associated with $\mathcal{T}_h^j$ (see~\cite{TW:2005,XZ:1998} for details), 
$B$ is a full-rank matrix with entries $0$ and $\pm 1$ enforcing continuity along $\Gamma$, 
and $(\cdot,\cdot)$ denotes the Euclidean inner product. 
It is straightforward to observe that~\eqref{FETI_constrained} is equivalent to the restriction of the finite element discretization of~\eqref{Poisson} on the global mesh $\mathcal{T}_h$ to the interface $\Gamma$.

By enforcing the constraint $Bv = 0$ in~\eqref{FETI_constrained} through the introduction of a Lagrange multiplier $\lambda \in W$, where $W$ denotes the vector space of Lagrange multipliers, we obtain the following saddle point problem, which is of the form~\eqref{model_saddle}:
\begin{equation}
\label{FETI_saddle}
\begin{bmatrix} A & B^t \\ B & 0 \end{bmatrix}
\begin{bmatrix} u \\ \lambda \end{bmatrix}
= \begin{bmatrix} f \\ 0 \end{bmatrix}.
\end{equation}
Note that $A$ is semi-SPD, owing to subdomains $\Omega_j$ that do not intersect $\partial \Omega$.
Finally, by eliminating $u$ in~\eqref{FETI_saddle} as in \cref{Sec:Saddle}, we obtain the following SPD linear system corresponding to~\eqref{Schur_singular}:
\begin{equation}
\label{FETI}
    S_0 \lambda_0 = d_0,
    \quad \text{where} \quad
    S_0 = PB A^+ (PB)^t \colon W_0 \to W_0,\ 
    d_0 \in W_0,
\end{equation}
and $W_0$ and $P$ are defined in the same way as in~\eqref{W0} and~\eqref{P}, respectively.

Now we estimate the condition number of $S_0$ in~\eqref{FETI} using \cref{Thm:Schur_singular}. 
We begin by summarizing some useful properties of the matrices $A$ and $B$. 
We have the following estimates for the extremal eigenvalues of $A$ (cf.~\cite[equations~(29) and~(30)]{FMR:1994}).

\begin{lemma}
\label{Lem:FETI_S}
The matrix $A$ defined in~\eqref{discrete_harmonic} satisfies
\begin{equation}
\label{FETI_S}
    \lambda_{\min}(A) \gtrsim \frac{h}{H}, 
    \quad
    \lambda_{\max}(A) \lesssim 1,
\end{equation}
where $\lambda_{\min}$ denotes the smallest nonzero eigenvalue.
\end{lemma}
\begin{proof}
For any $v = (v_j)_{j=1}^J \in \mathcal{R}(A)$, we have
\begin{equation*}
(A v, v)
\stackrel{\eqref{discrete_harmonic}}{\eqsim} \sum_{j=1}^J | v_j |_{H^{\frac{1}{2}} (\partial \Omega_j)}^2
\gtrsim \frac{1}{H} \sum_{j=1}^J \| v_j \|_{L^2 (\partial \Omega_j)}^2 
\eqsim \frac{h}{H} (v, v),
\end{equation*}
where $\gtrsim$ follows from the Poincar\'{e} inequality applied to each subdomain, and $\eqsim$ follows from standard scaling arguments~\cite[Section~3.4]{TW:2005}.
In particular, since $v \in \mathcal{R}(A) = \mathcal{N}(A)^{\perp}$, the Poincar\'{e} inequality also applies to floating subdomains $\Omega_j$ that do not intersect $\partial \Omega$.
On the other hand, the inverse inequality yields
\begin{equation*}
(A v, v)
\stackrel{\eqref{discrete_harmonic}}{\eqsim} \sum_{j=1}^J | v_j |_{H^{\frac{1}{2}} (\partial \Omega_j)}^2
\lesssim \frac{1}{h} \sum_{j=1}^J \| v_j \|_{L^2 (\partial \Omega_j)}^2 
\eqsim (v, v),
\end{equation*}
which completes the proof.
\end{proof}

\begin{remark}
\label{Rem:FETI_S}
Let $\hat{V}$ be a subspace of $V$ consisting of functions that are continuous across $\Gamma$.
Then the restriction $A|_{\hat{V}}$ is SPD, since continuity across $\Gamma$ eliminates the piecewise constant kernel modes associated with floating subdomains.
For $A|_{\hat{V}}$, we cannot apply the Poincar\'{e} inequality on each $\partial \Omega_j$ as in \cref{Lem:FETI_S}; instead, we obtain the following weaker estimate for the minimum eigenvalue (see, e.g.,~\cite[Lemma~4.11]{TW:2005} and~\cite[Lemma~3.1]{XZ:1998}), derived using the global Poincar\'{e} inequality on $\Omega$:
\begin{equation*}
(Av, v) \gtrsim H \sum_{j=1}^J \| v_j \|_{L^2 (\partial \Omega_j)}^2 \eqsim H h (v, v),
\quad v = (v_j)_{j=1}^J \in \hat{V}.
\end{equation*}
\end{remark}

In addition, as shown in~\cite[equation~(22)]{FMR:1994}, since the matrix $B$ consists only of entries $0$ and $\pm 1$, an elementary calculation gives
\begin{equation}
\label{FETI_B}
    \lambda_{\min}(B B^t) \gtrsim 1, 
    \quad
    \lambda_{\max}(B B^t) = \lambda_{\max}(B^t B) \lesssim 1.
\end{equation}

Following \cref{Rem:Schur_insight}, we seek a right inverse of $PB$ with range in $\mathcal{R}(A)$.
For any $\mu\in W_0$, we choose
\begin{equation*}
v = (PB)^t (PB(PB)^t)^{-1} \mu \in \mathcal{R} (A),
\end{equation*}
which satisfies $PBv = \mu$.
Here, the operators $(PB)^t$ and $PB (PB)^t$ are regarded as
\begin{equation*}
(PB)^t \colon W_0 \to \mathcal{R} (A),
\quad
PB (PB)^t \colon W_0 \to W_0,
\end{equation*}
in view of~\eqref{domain_restriction}.
By \cref{Thm:Schur_singular}, the minimum eigenvalue of $S_0$ is estimated as
\begin{multline*}
\lambda_{\min} (S_0)
    = \inf_{0 \neq \mu \in W_0} \sup_{v \in \mathcal{R} (A),\ PBv = \mu} \frac{\| \mu \|^2}{(A v, v)}
    \stackrel{\eqref{FETI_S}}{\gtrsim} \inf_{0 \neq \mu \in W_0} \sup_{v \in \mathcal{R} (A),\ PBv = \mu} \frac{\| \mu \|^2}{\| v \|^2} \\
    \geq \inf_{0 \neq \mu \in W_0} \frac{\| \mu \|^2}{\| (PB)^t (PB(PB)^t)^{-1} \mu \|^2}
    = \inf_{0 \neq \mu \in W_0} \frac{ \| \mu 
    \|^2}{( (PB(PB)^t)^{-1} \mu, \mu)}
    \stackrel{\eqref{FETI_B}}{\gtrsim} 1.
\end{multline*}
On the other hand, the maximum eigenvalue of $S_0$ is estimated as
\begin{equation}
\label{S_0_max}
    \lambda_{\max} (S_0)
    = \sup_{0 \neq v \in \mathcal{R} (A)} \frac{\| P B v \|^2}{(A v, v)}
    \stackrel{\eqref{FETI_B}}{\lesssim} \sup_{0 \neq v \in \mathcal{R} (A)} \frac{\| v \|^2}{(A v, v)}
    = \lambda_{\min} (A)^{-1}
    \stackrel{\eqref{FETI_S}}{\lesssim} \frac{H}{h}.
\end{equation}
Consequently, we get
\begin{equation*}
    \kappa (S_0) \lesssim \frac{H}{h}.
\end{equation*}
This result agrees with~\cite[Theorem~3.2]{FMR:1994}.

Next, we consider the following scaled Dirichlet preconditioner:
\begin{equation*}
    L_{\mathrm{FETI}} = P B_D A (P B_D)^t,
\end{equation*}
where $B_D = D_{\mathrm{FETI}} B$ is a scaled jump matrix~\cite{SM:2008,TW:2005}, with $D_{\mathrm{FETI}}$ a block-diagonal scaling matrix chosen such that $B B_D^t = I$. 
In other words, $B_D^t$ is a right inverse of $B$. 
In this case, we can apply \cref{Cor:Schur_singular_preconditioned} to estimate the extremal eigenvalues of the preconditioned operator $L_{\mathrm{FETI}} S_0$.

The only technical ingredient we require is the well-known vertex--edge lemma~\cite[Section~4.3]{XZ:1998}, stated in \cref{Lem:vertex_edge}. 
In what follows, edges are understood as relatively open subsets, i.e., without their endpoints. 
For a subset $K \subset \partial \Omega_j$, the restriction $I_K^0 v_j \in V_j$ of $v_j \in V_j$ is defined by
\begin{equation}
\label{restriction}
I_K^0 v_j = 
\begin{cases}
v_j & \text{at nodal points in } K, \\
0   & \text{at all other nodal points}.
\end{cases}
\end{equation}
In \cref{Lem:vertex_edge}, we use the scaled norm
\begin{equation*}
\| v_j \|_{H^{\frac{1}{2}} (\partial \Omega_j)}^2
= H^{-1} \| v_j \|_{L^2 (\partial \Omega_j)}^2
+ | v_j |_{H^{\frac{1}{2}} (\partial \Omega_j)}^2.
\end{equation*}

\begin{lemma}
\label{Lem:vertex_edge}
For each subdomain $\Omega_j \subset \mathbb{R}^2$, the following estimates hold:
\begin{enumerate}[(a)]
\item If $x$ is a vertex of $\partial \Omega_j$, then we have
\begin{equation*}
    \| I_x^0 v_j \|_{H^{\frac{1}{2}} (\partial \Omega_j)}^2 \lesssim \left( 1 + \log \frac{H}{h} \right) \| v_j \|_{H^{\frac{1}{2}} (\partial \Omega_j)}^2, \quad v_j \in V_j.
\end{equation*}
\item If $e$ is an edge of $\partial \Omega_j$, then we have
\begin{equation*}
    \| I_e^0 v_j \|_{H^{\frac{1}{2}} (\partial \Omega_j)}^2 \lesssim \left( 1 + \log \frac{H}{h} \right)^2 \| v_j \|_{H^{\frac{1}{2}} (\partial \Omega_j)}^2, \quad v_j \in V_j.
\end{equation*}
\end{enumerate}
\end{lemma}

Now we are ready to analyze the preconditioned operator $L_{\mathrm{FETI}} S_0$. 
By \cref{Cor:Schur_singular_preconditioned}, we have 
\begin{equation*}
\lambda_{\min}(L_{\mathrm{FETI}} S_0) \geq 1.
\end{equation*}
To estimate the maximum eigenvalue of $L_{\mathrm{FETI}} S_0$, \cref{Cor:Schur_singular_preconditioned} with $\bar{B} = B_D$ implies that it suffices to bound $( A (P B_D)^t PB v, (P B_D )^t PB v)$ for any $v \in \mathcal{R} (A)$. 
Following \cref{Rem:Schur_insight}, we seek a correction in the null space of the projected constraint $PB$.
As observed in the proof of \cref{Prop:Schur_singular}, $\mathcal{N}(A)\subset\mathcal{N}(PB)$ when $PB$ is regarded as an operator on $V$, and a correction in $\mathcal{N}(A)$ also preserves the energy $(Av,v)$.
Take any $v \in \mathcal{R} (A)$ and choose $\phi \in \mathcal{N}(A)$ such that $B (v+\phi) \in W_0$. 
Then we have
\begin{equation}
\label{FETI_phi}
P^t P B(v+\phi) = B(v+\phi).
\end{equation}
Since $PB\phi = 0$, the identity~\eqref{FETI_phi} gives
\begin{equation}
\label{FETI_phi_bound}
B\phi = -(I-P^tP)Bv,
\quad
\| B\phi \| \leq \| Bv \| \stackrel{\eqref{FETI_B}}{\lesssim} \| v \|.
\end{equation}
We estimate the contributions of $v$ and $\phi$ separately.

First, by \cref{Lem:vertex_edge} and the local Poincar\'{e} inequality used in the proof of \cref{Lem:FETI_S}, we obtain
\begin{equation}
\label{FETI_v_energy}
\begin{split}
(A B_D^t Bv, B_D^t Bv)
&\lesssim \left( 1 + \log \frac{H}{h} \right)^2
\sum_{j=1}^J \left(
H^{-1} \| v_j \|_{L^2 (\partial \Omega_j)}^2
+ | v_j |_{H^{\frac{1}{2}} (\partial \Omega_j)}^2 \right) \\
&\lesssim \left( 1 + \log \frac{H}{h} \right)^2 (Av,v),
\end{split}
\end{equation}
where we have used $v \in \mathcal{R}(A)$.

Next, since $\phi \in \mathcal{N}(A)$, we identify each $\phi_j$ with its constant value.
The edge and crosspoint values of $B_D^t B\phi$ are controlled by jumps across neighboring interface edges, so \cref{Lem:vertex_edge} gives
\begin{multline}
\label{FETI_phi_energy}
(A B_D^t B\phi, B_D^t B\phi)
\lesssim \left( 1 + \log \frac{H}{h} \right)^2
\sum_{\Gamma_{ij}\text{: interface edge}} |\phi_i-\phi_j|^2 \\
\lesssim \frac{h}{H} \left( 1 + \log \frac{H}{h} \right)^2 \| B\phi \|^2
\lesssim \left( 1 + \log \frac{H}{h} \right)^2 (Av,v).
\end{multline}
Here the second inequality follows by counting the fine-grid nodes on each interface edge, and the last follows from~\eqref{FETI_S} and~\eqref{FETI_phi_bound}.

Combining~\eqref{FETI_v_energy} and~\eqref{FETI_phi_energy} with~\eqref{FETI_phi}, we obtain
\begin{equation*}
\begin{split}
(A (P B_D)^t PBv, (P B_D)^t PBv)
&= (A B_D^t B(v+\phi), B_D^t B(v+\phi)) \\
&\lesssim (A B_D^t Bv, B_D^t Bv)
+ (A B_D^t B\phi, B_D^t B\phi) \\
&\lesssim \left( 1 + \log \frac{H}{h} \right)^2 (Av,v).
\end{split}
\end{equation*}
Here we require mild geometric assumptions on the domain decomposition $\{ \Omega_j \}_{j=1}^J$, namely that the numbers of vertices and edges of the subdomains are uniformly bounded and that the number of subdomains sharing a common vertex is also uniformly bounded. 
Hence, by \cref{Cor:Schur_singular_preconditioned}, we deduce that
\begin{equation}
\label{LS_0_max}
    \lambda_{\max}(L_{\mathrm{FETI}} S_0) \lesssim \left( 1 + \log \frac{H}{h} \right)^2.
\end{equation}

In summary, we obtain the following condition number bound for the preconditioned operator $L_{\mathrm{FETI}} S_0$ (cf.~\cite[Theorem~6.15]{TW:2005}):
\begin{equation*}
\kappa(L_{\mathrm{FETI}} S_0) \lesssim \left( 1 + \log \frac{H}{h} \right)^2.
\end{equation*}

\begin{remark}
\label{Rem:FETI_separate}
Estimating the factors of $L_{\mathrm{FETI}} S_0$ separately gives a weaker bound.
Indeed, by~\eqref{FETI_S} and~\eqref{FETI_B}, we have
\begin{equation*}
\lambda_{\max}(L_{\mathrm{FETI}})
\leq \lambda_{\max}(A)\|(PB_D)^t\|^2
\leq \lambda_{\max}(A)\|B_D\|^2
\lesssim 1.
\end{equation*}
Combining this with~\eqref{S_0_max} gives
\begin{equation*}
\lambda_{\max}(L_{\mathrm{FETI}}S_0)
\leq \lambda_{\max}(L_{\mathrm{FETI}})\lambda_{\max}(S_0)
\lesssim \frac{H}{h}.
\end{equation*}
On the other hand, the variational estimates presented in this paper retain the interaction between $L_{\mathrm{FETI}}$ and $S_0$, which we exploit through the null space correction in~\eqref{FETI_phi} to obtain the logarithmic bound~\eqref{LS_0_max}.
\end{remark}

\begin{remark}
\label{Rem:FETI_scaling}
The original Dirichlet preconditioner~\cite{FMR:1994,MT:1996} uses the unscaled jump matrix $B$ and has the projected form $PBA(PB)^t$ in our notation.
From the viewpoint of \cref{Rem:Schur_insight}, the scaled jump matrix $B_D$ used in~\cite{MS:2007,SM:2008} is a natural choice because $B_D^t$ is a right inverse of $B$.
\end{remark}

\subsection{FETI-DP}
Next, we consider FETI-DP~\cite{FLLPR:2001,KWD:2002}, 
which is one of the most widely used variants of FETI. 
It has been applied not only to linear problems but also to more complex nonlinear problems arising in computational mechanics~\cite{KLR:2014,LP:2021a}.

While the FETI formulation~\eqref{FETI_constrained} enforces continuity along the entire interface $\Gamma$ through the constraint, FETI-DP instead imposes continuity at subdomain corners directly by restricting the solution space to a subspace $\tilde{V} \subset V$. 
The space $\tilde{V}$ consists of functions in $V$ that are continuous at subdomain corners.
Continuity along the interior of each subdomain edge is enforced by constraints. 
Namely, we consider the following formulation:
\begin{equation}
\label{FETI-DP_constrained}
\min_{v = (v_j)_{j=1}^J \in \tilde{V}} \left\{ \frac{1}{2} (\tilde{A} v, v) - (\tilde{f}, v) \right\},
\quad \text{subject to} \quad
B v = 0,
\end{equation}
where the matrix $\tilde{A}$ and the vector $\tilde{f}$ are defined analogously to~\eqref{discrete_harmonic}; see~\cite{KWD:2002,TW:2005} for details. 
Here $B$ is a full-rank matrix with entries $0$ and $\pm 1$ that enforces continuity along the interior of subdomain edges, and $(\cdot,\cdot)$ denotes the Euclidean inner product. 

By enforcing the constraint $B v = 0$ in~\eqref{FETI-DP_constrained} through the introduction of a Lagrange multiplier $\lambda \in W$, where $W$ denotes the vector space of Lagrange multipliers, we obtain the saddle point problem
\begin{equation}
\label{FETI-DP_saddle}
\begin{bmatrix}
\tilde{A} & B^t \\ B & 0 
\end{bmatrix}
\begin{bmatrix}
    u \\ \lambda
\end{bmatrix}
= \begin{bmatrix}
    \tilde{f} \\ 0
\end{bmatrix}.
\end{equation}
In contrast to the FETI formulation~\eqref{FETI_saddle}, in~\eqref{FETI-DP_saddle} the matrix $\tilde{A}$ is SPD. 
Therefore, by eliminating $u$ in~\eqref{FETI-DP_saddle}, we obtain the following Schur complement system corresponding to~\eqref{Schur}:
\begin{equation}
\label{FETI-DP}
S \lambda = d, 
\quad \text{where} \quad
S = B \tilde{A}^{-1} B^t, 
\quad d = B \tilde{A}^{-1} \tilde{f}.
\end{equation}

In~\eqref{FETI-DP_constrained}, since the matrix $B$ enforces continuity only at nodal points located in the interior of subdomain edges, each nodal point is affected by $B$ only once. 
Hence, we have the following identity~\cite[equation~(3.2)]{MT:2001}:
\begin{equation}
\label{FETI-DP_B}
B B^t = 2I.
\end{equation}
Thus, if we define $\bar{B} = \tfrac{1}{2} B$, then $\bar{B}^t$ is a right inverse of $B$.

As FETI-DP involves functions that are continuous at subdomain corners, we require certain Poincar\'{e}-type inequalities associated with subdomain corners in order to analyze FETI-DP. 
In \cref{Lem:Poincare}, which summarizes results from~\cite[Lemma~2]{LPP:2021} and~\cite[Lemma~5.1]{MT:2001}, $L^2$ and $H^{\frac{1}{2}}$ estimates are provided for functions vanishing at subdomain corners. 
In what follows, let $I_H$ denote the interpolation operator whose restriction to each subdomain edge is linear and agrees with the given function at the endpoints.

\begin{lemma}
\label{Lem:Poincare}
For each subdomain $\Omega_j \subset \mathbb{R}^2$, the following estimates hold:
\begin{equation*}
\begin{aligned}
\| v_j - I_H v_j \|_{L^2 (\partial \Omega_j)}^2 \lesssim H \left( 1 + \log \frac{H}{h} \right) | v_j |_{H^{\frac{1}{2}} (\partial \Omega_j)}^2, \\
\sum_{e \text{:~edge of } \Omega_j} | I_e^0 (v_j - I_H v_j) |_{H^{\frac{1}{2}} (\partial \Omega_j)}^2 \lesssim \left( 1 + \log \frac{H}{h} \right)^2 | v_j |_{H^{\frac{1}{2}} (\partial \Omega_j)}^2,
\end{aligned}
\quad v_j \in V_j,
\end{equation*}
where $I_e^0$ was defined in~\eqref{restriction}.
\end{lemma}

Now, we estimate the extremal eigenvalues of $S$ given in~\eqref{FETI-DP}.
By \cref{Thm:Schur}, we obtain
\begin{equation*}
\lambda_{\min} (S) = \inf_{0 \neq \lambda \in W} \sup_{v \in \tilde{V},\ Bv = \lambda} \frac{\| \lambda \|^2}{(\tilde{A}v, v)}
\geq \inf_{0 \neq \lambda \in W} \frac{\| \lambda \|^2}{( \tilde{A} \bar{B}^t \lambda, \bar{B}^t \lambda)}
\stackrel{\eqref{FETI-DP_B}}{\gtrsim} \lambda_{\max} (\tilde{A})^{-1} 
\stackrel{\eqref{FETI_S}}{\gtrsim} 1.
\end{equation*}
For the maximum eigenvalue, \cref{Rem:Schur_insight} suggests choosing a suitable correction in $\mathcal{N}(B)$.
Since $\mathcal{R}(I_H)\subset\mathcal{N}(B)$, we choose $\phi=-I_Hv$ for any $v = (v_j)_{j=1}^J \in \tilde{V}$ and obtain
\begin{multline*}
\| B v \|^2
= \| B (v - I_H v ) \|^2
\stackrel{\eqref{FETI-DP_B}}{\lesssim} \| v - I_H v \|^2
\eqsim h^{-1} \sum_{j=1}^J \| v_j - I_H v_j \|_{L^2 (\partial \Omega_j)}^2 \\
\lesssim \frac{H}{h} \left( 1 + \log \frac{H}{h} \right) \sum_{j=1}^J | v_j |_{H^{\frac{1}{2}} (\partial \Omega_j)}^2
\eqsim \frac{H}{h} \left( 1 + \log \frac{H}{h} \right) (\tilde{A} v, v),
\end{multline*}
where the last inequality is due to \cref{Lem:Poincare}.
Hence, \cref{Thm:Schur} implies that
\begin{equation*}
    \lambda_{\max} (S) \lesssim \frac{H}{h} \left( 1 + \log \frac{H}{h} \right).
\end{equation*}
Consequently, we deduce that
\begin{equation*}
\kappa (S) \lesssim \frac{H}{h} \left( 1 + \log \frac{H}{h} \right),
\end{equation*}
which agrees with the existing result in~\cite[Proposition~3]{LPP:2021}.

Next, we consider the following Dirichlet preconditioner:
\begin{equation*}
L_{\mathrm{DP}} = \bar{B} \tilde{A} \bar{B}^t.
\end{equation*}
By \cref{Cor:Schur_preconditioned}, we have $\lambda_{\min} (L_{\mathrm{DP}} S) \geq 1$.
Using the same null space correction, for any $v = (v_j)_{j=1}^J \in \tilde{V}$, it follows that
\begin{equation*}
\begin{split}
(\tilde{A} \bar{B}^t B v, \bar{B}^t B v)
&= ( \tilde{A} \bar{B}^t B (v - I_H v), \bar{B}^t B (v - I_H v)) \\
&\eqsim \sum_{j=1}^J | ( \bar{B}^t B (v - I_H v) )_j |_{H^{\frac{1}{2}} (\partial \Omega_j)}^2 \\
&\lesssim \sum_{j=1}^J \sum_{e \text{:~edge of } \Omega_j} \left| I_e^0 (v_j - I_H v_j ) \right|_{H^{\frac{1}{2}} (\partial \Omega_j)}^2 \\
&\lesssim \left( 1 + \log \frac{H}{h} \right)^2 \sum_{j=1}^J | v_j |_{H^{\frac{1}{2}} (\partial \Omega_j)}^2 \\
&\eqsim \left(1 + \log \frac{H}{h} \right)^2 (\tilde{A}v, v),
\end{split}
\end{equation*}
where the first inequality follows from the definitions of $B$ and $\bar{B}$ together with the triangle inequality, and the second inequality is due to \cref{Lem:Poincare}.
This implies
\begin{equation*}
    \lambda_{\max} (L_{\mathrm{DP}} S) \lesssim \left(1 + \log \frac{H}{h} \right)^2.
\end{equation*}

Finally, we obtain the following condition number bound for the preconditioned operator $L_{\mathrm{DP}} S$ (cf.~\cite[Theorem~5.2]{MT:2001}):
\begin{equation*}
    \kappa(L_{\mathrm{DP}} S) \lesssim \left(1 + \log \frac{H}{h} \right)^2.
\end{equation*}
We note that analyses of three-dimensional problems and heterogeneous coefficients can be found in~\cite{KWD:2002,TW:2005}.

\begin{remark}
\label{Rem:BDDC}
Two nonoverlapping domain decomposition methods closely related to FETI and FETI-DP are BDD (balancing domain decomposition)~\cite{FP:2003,Mandel:1993} and BDDC (balancing domain decomposition by constraints)~\cite{Dohrmann:2003,FP:2003}. 
It has been shown in the literature~\cite{BS:2007,LW:2006,MDT:2005,MS:2007,SM:2008} that FETI/FETI-DP with scaled Dirichlet preconditioners and BDD/BDDC are spectrally equivalent in the sense that their spectra agree except for the eigenvalues $0$ and $1$. 
In fact, BDD and BDDC can be analyzed within our framework (\cref{Thm:Schur_preconditioned,Thm:Schur_singular_preconditioned}) by essentially the same arguments used for FETI and FETI-DP. 
For brevity, we omit the details.
\end{remark}

\section{Concluding remarks}
\label{Sec:Conclusion}
In this paper, we have presented sharp estimates for the Schur complements arising in symmetric saddle point problems.
The resulting expressions make explicit the roles of right inverses with controlled energy and null space corrections in estimating the extremal eigenvalues.
The applications illustrate how these choices recover convergence estimates for the augmented Lagrangian method, mixed finite element methods, and nonoverlapping domain decomposition methods.

We conclude the paper with two remarks. 
The first concerns the case when the operator $B$ in~\eqref{model_saddle} is not surjective, so that~\eqref{model_saddle} becomes singular and may admit nonunique solutions. 
Such a situation arises, for example, when redundant Lagrange multipliers are used to enforce the constraint $Bv = g$ in~\eqref{model_saddle_opt} (see, e.g.,~\cite{Pechstein:2013,TW:2005}). 
In this case, the Schur complement $S$ may become semi-SPD. 
Although we do not treat this case in detail here for the sake of simplicity and readability, we note that it can be handled by essentially the same arguments presented in this paper~(cf.~\cite{PX:2025}).

The second remark concerns the extension of the variational arguments to convex optimization problems. 
As mentioned at several points in this paper (see, e.g.,~\cref{Rem:FR}), some of the results in this paper can already be extended to more general classes of convex optimization problems, although a full generalization is highly nontrivial. 
We note that there has recently been significant research on extending the theory of iterative methods for linear systems to convex optimization problems; see~\cite{JPX:2025} and the references therein. 
In this spirit, we view the extension of our results to convex optimization problems as an important direction for future work, which we expect to be highly impactful.

\bibliographystyle{siamplain}
\bibliography{refs_aux}

\end{document}